\documentclass[12pt,a4paper]{amsart}

\usepackage[T1]{fontenc}
\usepackage[utf8]{inputenc}
\usepackage{lmodern}
\usepackage[english]{babel}
\usepackage{amsmath,amssymb,amsthm,mathtools}
\usepackage{geometry}
\usepackage{microtype}
\usepackage[unicode,hidelinks]{hyperref}

\newcommand{\bP}{\mathbb P}
\newcommand{\bZ}{\mathbb Z}
\newcommand{\bF}{\mathbb F}
\newcommand{\kk}{\Bbbk}
\newcommand{\cO}{\mathcal O}
\newcommand{\Oq}{\mathcal O_q}

\DeclareMathOperator{\Ext}{Ext}
\DeclareMathOperator{\Hom}{Hom}
\DeclareMathOperator{\RHom}{RHom}
\DeclareMathOperator{\Cone}{Cone}
\DeclareMathOperator{\rk}{rk}

\theoremstyle{plain}
\newtheorem{lemma}{Lemma}
\newtheorem{proposition}{Proposition}
\newtheorem{theorem}{Theorem}

\theoremstyle{definition}

\newtheorem{remark}{Remark}

\title[Nontransitivity of braid group actions]{Nontransitivity of braid group
actions on exceptional bases in $K_0(\bP^n)$ for $n=4,6$}
\author{R.~V.~Eliseev}
\date{}

\makeatletter
\let\shorttitleWithOriginalCase\shorttitle
\def\@maketitle@hook{%
  \global\let\shorttitle\shorttitleWithOriginalCase
  \global\let\@maketitle@hook\@empty
}
\def\@settitle{%
  \begin{center}%
    \baselineskip 14\p@\relax
    \normalfont\bfseries\Large
    \@title
  \end{center}%
}
\makeatother

\begin{document}

\begin{abstract}
Let $X=\bP^n$, where $p=n+1$ is prime. Consider the orbit of the standard
numerical exceptional basis in $K_0(X)$ under mutations, sign changes, and
isometries of the Euler form. We construct an invariant of this orbit: for every basis from the orbit, the square of rank of all basic vectors art congruent to $1$
modulo $p$. For $\bP^4$ and $\bP^6$, we present explicit numerical
exceptional bases violating this congruence; hence the action is not
transitive. We also establish a stronger obstruction modulo $p^2$ and use
Polishchuk's theorem to show that no vector of the constructed bases is the
class of an exceptional object. Finally, we prove that for every prime $p$, the
category $D^b(\operatorname{Coh}\bP^{p-1}_{\mathbb C})$ does not contain a
pair of mutually orthogonal exceptional objects.
\end{abstract}

\maketitle

\section*{Introduction}

Let $X$ be a smooth projective variety over a field $\kk$ and let
$D^b(\operatorname{Coh}X)$ be the bounded derived category of coherent sheaves
on $X$. An object $E\in D^b(\operatorname{Coh}X)$ is called \emph{exceptional}
if
\[
\dim_\kk\Hom(E,E)=1,
\qquad
\Ext^k(E,E)=0 \quad \text{for } k\ne0.
\]
An ordered collection $(E_0,\dots,E_n)$ of exceptional objects is called
\emph{exceptional} if
\[
\Ext^k(E_j,E_i)=0 \quad \text{for all } k \text{ and all } i<j.
\]
An exceptional collection generating the whole category is called
\emph{full}. A classical theorem of Beilinson~\cite{Beilinson} states that
$(\cO,\cO(1),\dots,\cO(n))$ is a full exceptional collection on $\bP^n$.

The \emph{left mutation} replaces a pair $(E_i,E_{i+1})$ by the pair
$(L_{E_i}E_{i+1},E_i)$, where
\[
L_{E_i}E_{i+1}
=
\Cone\bigl(\RHom(E_i,E_{i+1})\otimes E_i\to E_{i+1}\bigr).
\]
The \emph{right mutation} replaces $(E_i,E_{i+1})$ by the pair
$(E_{i+1},R_{E_{i+1}}E_i)$, where
\[
R_{E_{i+1}}E_i
=
\Cone\bigl(E_i\to
\RHom(E_i,E_{i+1})^{\vee}\otimes E_{i+1}\bigr)[-1].
\]
Mutations define an action of the braid group on exceptional collections of
fixed length. One of the central questions of the theory is whether this
action is transitive (see, for example,~\cite{GorodentsevKuleshovHelix}), i.e.,
whether every full exceptional collection can be obtained from the standard
one by a sequence of mutations and autoequivalences of the derived category.

In this paper we study a numerical version of this question for projective
spaces. Consider the Euler form on the Grothendieck lattice $K_0(\bP^n)$:
\begin{equation}\label{eq:euler-form}
\langle x,y\rangle=\chi(x,y)
=
\sum_{i\geq0}(-1)^i\dim_\kk\Ext^i(x,y).
\end{equation}
A basis $(e_0,\dots,e_n)$ of this lattice is called a \emph{numerical
exceptional basis} if the Gram matrix of the Euler form in this basis is upper
unitriangular, i.e., $\langle e_i,e_i\rangle=1$ for all
$i$ and $\langle e_i,e_j\rangle=0$ for $i>j$. The classes of a full exceptional
collection form such a basis. However, a numerical exceptional basis need not
come from a collection of objects. At the level of $K_0$, the previous operations turn to mutations and the sign changes of basic vectors, and the isometries of the Euler form: mutations, sign changes of basis vectors,
and isometries of the Euler form. The standard basis
$([\cO],[\cO(1)],\dots,[\cO(n)])$ determines a distinguished orbit of this
action, called the \emph{standard orbit}. The question is whether the standard
orbit contains all numerical exceptional bases.

Our main result is a negative answer for $\bP^4$ and $\bP^6$. The key
observation is as follows. Suppose $p=n+1$ is prime; then every numerical
exceptional basis $(e_0,\dots,e_n)$ in the standard orbit satisfies
\[
\rk(e_i)^2\equiv1\pmod p,
\qquad i=0,\dots,n.
\]
This congruence is an obstruction to transitivity. We present explicit
numerical exceptional bases in $K_0(\bP^4)$ and $K_0(\bP^6)$ violating this
congruence for some basis vector. This proves that the action of mutations,
sign changes, and isometries of the Euler form on numerical exceptional bases
is not transitive.

We also find a second obstruction, independent of the rank invariant. Namely,
in every basis of the standard orbit, all entries above the of the Gram matrix
diagonal are nonzero modulo $p^2$. Therefore any basis from the standard orbit does not contain mutually orthogonal vectors, i.e., vectors $x,y$ such that
$\langle x,y\rangle=\langle y,x\rangle=0$. On the other hand, the counterexample
in $K_0(\bP^4)$ contains such pairs.

A.~Polishchuk~\cite[Theorem~1.3(iii)]{Polishchuk} obtained a similar
restriction for actual exceptional object on the level of $D^b(\operatorname{Coh}\bP^{p-1}_{\mathbb C})$ exceptional objects. In particular, if $p$ is prime,
then modulo $p$ the class of an exceptional object on $\bP^{p-1}_{\mathbb C}$
coincides with the class of a line bundle up to sign. It follows from this
restriction and our computations that every vector of our counterexamples is  not a
class of an exceptional object in the derived category. Moreover, we prove
that $D^b(\operatorname{Coh}\bP^{p-1}_{\mathbb C})$ does not contain a pair of
mutually orthogonal exceptional objects, i.e., exceptional objects $E,F$ such
that $\Ext^\bullet(E,F)=0$ and $\Ext^\bullet(F,E)=0$.

\subsection*{Acknowledgments}
The author is grateful to A.~L.~Gorodentsev and D.~V.~Pirozhkov for valuable
comments, useful discussions, and attention to this work.

\section{Arithmetic properties of the standard orbit}\label{sec:orbit}

Throughout this section, let $X=\bP^n$, $p=n+1$, and $K=K_0(X)$. Fix a point
$q\in X$ and put
\[
\rho(x)=\langle x,[\Oq]\rangle,
\qquad x\in K.
\]
For $X=\bP^n$, the function $\rho$ coincides with the rank: $\rho=\rk$. Since
$\rk\cO(i)=1$ for all $i$, the rank of a vector equals the sum of its
coordinates in the standard basis $([\cO],[\cO(1)],\dots,[\cO(n)])$.

\begin{lemma}\label{lem:std-mod-p}
Let $p=n+1$ be prime. For the standard collection on $\bP^n$ and
$0\le i\le j\le n$, we have
\[
\langle \cO(i),\cO(j)\rangle
=
\chi(\cO(j-i))
=
\binom{n+j-i}{j-i}.
\]
If $i<j$, then this number is divisible by $p$. Consequently, the Gram matrix
$B_{\mathrm{std}}$ of the standard collection satisfies
\[
B_{\mathrm{std}}\equiv I_{n+1}\pmod p.
\]
\end{lemma}

\begin{proof}
Suppose $i<j$ and put $d=j-i$; then $1\le d\le n$ and
\[
\binom{n+d}{d}=\binom{n+d}{n}=\frac{(n+d)!}{d!\,n!}.
\]
The numerator $(n+d)!$ is divisible by $p=n+1$. Since $p$ is prime and
$1\le d,n<p$, the denominator $d!\,n!$ is not divisible by $p$. Therefore,
\[
\binom{n+d}{d}\equiv0\pmod p.
\]
The diagonal entries of $B_{\mathrm{std}}$ are equal to $1$. The entries below
the diagonal vanish because the standard collection is exceptional.
\end{proof}

\begin{lemma}\label{lem:mutation-rank}
Let $\mathcal E=(e_0,\dots,e_n)$ be a numerical exceptional basis such that
\[
\langle e_i,e_j\rangle\equiv\delta_{ij}\pmod p.
\]
Then a mutation of any pair preserves this congruence and does not
change the squared rank of any basis vector modulo $p$.
\end{lemma}

\begin{proof}
Consider a pair $(e_i,e_{i+1})$ and put
\[
m=\langle e_i,e_{i+1}\rangle.
\]
By assumption, $m\equiv0\pmod p$. The right mutation replaces this pair by
\[
(e_{i+1},f),
\qquad\text{where }
f=e_i-me_{i+1}.
\]
A sign change replaces $f$ by $-f$ and does not affect the squared rank of $f$.
The vector $e_{i+1}$ only changes its position in the pair, so the rank of
$e_{i+1}$ does not change. For the new vector $f$, we have
\[
\rho(f)=\rho(e_i)-m\rho(e_{i+1})\equiv\rho(e_i)\pmod p.
\]
Since $\rho=\rk$, it follows that
\[
\rk(f)^2\equiv\rk(e_i)^2\pmod p.
\]
Thus the squared rank of each basis vector remains the same modulo $p$.

It remains to check that the new Gram matrix is still congruent to the identity
matrix modulo $p$. If $k\ne i,i+1$, then
\[
\langle f,e_k\rangle
=
\langle e_i,e_k\rangle-m\langle e_{i+1},e_k\rangle
\equiv0\pmod p
\]
and similarly
\[
\langle e_k,f\rangle
=
\langle e_k,e_i\rangle-m\langle e_k,e_{i+1}\rangle
\equiv0\pmod p.
\]
Moreover,
\[
\langle e_{i+1},f\rangle
=
\langle e_{i+1},e_i\rangle-m\langle e_{i+1},e_{i+1}\rangle
=-m
\equiv0\pmod p
\]
and
\[
\langle f,e_{i+1}\rangle
=
\langle e_i,e_{i+1}\rangle-m\langle e_{i+1},e_{i+1}\rangle
=0.
\]
Finally,
\[
\langle f,f\rangle
=
\langle e_i-me_{i+1},e_i-me_{i+1}\rangle
=1-m^2+m^2=1.
\]
Thus after the mutation all off-diagonal entries of the Gram matrix are still
divisible by $p$ and all diagonal entries are equal to $1$. The left (inverse)
mutation is treated in the same way.
\end{proof}

\begin{lemma}\label{lem:isometry-rank}
Let $X=\bP^n$, where $p=n+1$ is prime. Then sign changes of basis vectors and
isometries of the Euler form do not change the squared rank of any basis vector
modulo $p$.
\end{lemma}

\begin{proof}
A sign change $e_i\mapsto-e_i$ replaces $\rk(e_i)$ by $-\rk(e_i)$ and therefore
preserves $\rk(e_i)^2$.

Define the Serre operator $S\colon K\to K$ by the identity
$\langle u,w\rangle=\langle w,Su\rangle$ for all $u,w\in K$. By Serre duality,
$\chi(u,w)=(-1)^n\chi(w,u\otimes\omega)$; hence,
\[
S=(-1)^n m_{t^{-p}},
\]
where $t=[\cO(1)]$ and $m_x\colon K\to K$ denotes multiplication by $x$.
Further, to each operator $f\colon K\to K$ assign the adjoint operator $f^*$
defined by $\langle f^*u,w\rangle=\langle u,fw\rangle$. For a nonsymmetric
form, the operation $f\mapsto f^*$ is no longer an involution: a direct
computation shows that
\[
f^{**}=S^{-1}fS.
\]
Hence $f^{**}=f$ if and only if $f$ commutes with $S$. Operators $f$ such that
$f^{**}=f$ are called \emph{reflexive}; thus every reflexive operator preserves
the eigenspaces of $S$.

Let us show that on $\bP^n$ the operator $m_t$ is a single unipotent Jordan
block of size $p$; in particular, the unique eigenspace of $m_t$ is
one-dimensional. We also show that this eigenspace is spanned by the class of
a point. Indeed, put $u=t-1$. Then $u$ is nilpotent, $u^{n+1}=0$, and the class
of a point is expressed in terms of $u$ as follows:
\[
[\Oq]=(1-t^{-1})^n=t^{-n}u^n=u^n.
\]
Hence $[\Oq]$ is a primitive vector of the lattice $K$. The element $t^{-p}$ is
unipotent, and
\[
t^{-p}-1=-pu+\text{terms of degree at least }2.
\]
Suppose a nonzero element $x\in K$ has lowest $u$-degree $r<n$; then the lowest
term of the product $(t^{-p}-1)x$ is $-pa_ru^{r+1}$, where $a_r\ne0$ is the
coefficient of $u^r$ in $x$. Therefore this product is nonzero and
\[
\ker\bigl(S-(-1)^n\bigr)
=
\ker m_{t^{-p}-1}
=
\bZ\,u^n
=
\bZ[\Oq].
\]
This eigenspace is one-dimensional and its intersection with the lattice is
exactly $\bZ[\Oq]$.

Now let $g$ be an isometry of the Euler form. Rewriting the invariance of the
form as $\langle u,gw\rangle=\langle g^{-1}u,w\rangle$, we get $g^*=g^{-1}$;
the same identity for the isometry $g^{-1}$ gives $g^{**}=g$. Thus the operator
$g$ is reflexive and hence commutes with $S$. Consequently, $g$ preserves the
eigenline $\bZ[\Oq]$. Since $g$ is an automorphism of the lattice $K$, the
operator $g$ takes the primitive generator $[\Oq]$ of this line to a primitive
generator; hence $g[\Oq]=\varepsilon[\Oq]$ for some $\varepsilon=\pm1$. Then
for each $i$ we have
\[
\rk(ge_i)=\langle ge_i,[\Oq]\rangle
=\varepsilon\langle ge_i,g[\Oq]\rangle
=\varepsilon\langle e_i,[\Oq]\rangle
=\varepsilon\rk(e_i),
\]
and hence $\rk(ge_i)^2=\rk(e_i)^2$. Thus isometries preserve the squared
rank of each vector exactly, not only modulo $p$. Together with the case of sign
changes, this proves the lemma.
\end{proof}

\begin{proposition}[rank invariant]\label{prop:rank-invariant}
Let $X=\bP^n$, where $p=n+1$ is prime. Suppose a numerical exceptional basis
$\mathcal E=(e_0,\dots,e_n)$ lies in the orbit of the standard basis
\[
([\cO],[\cO(1)],\dots,[\cO(n)])
\]
under mutations, sign changes, and isometries of the Euler form; then
\[
\rk(e_i)^2\equiv1\pmod p,
\qquad i=0,\dots,n.
\]
\end{proposition}

\begin{proof}
By Lemma~\ref{lem:std-mod-p}, the standard basis satisfies the congruence
$B_{\mathrm{std}}\equiv I_{n+1}\pmod p$. Each vector of the standard basis has
rank $1$; hence $\rk(\cO(i))^2\equiv1\pmod p$ for all $i$.

Each of the three allowed operations preserves the congruence
$\langle e_i,e_j\rangle\equiv\delta_{ij}\pmod p$. For mutations, this is
proved in Lemma~\ref{lem:mutation-rank}. A sign change multiplies one row and
one column of the Gram matrix by $-1$. An isometry does not change the Gram
matrix at all. Therefore the assumption of Lemma~\ref{lem:mutation-rank} holds
throughout the orbit. By Lemma~\ref{lem:mutation-rank}, mutations do not
change the squared rank of any basis vector modulo $p$. By
Lemma~\ref{lem:isometry-rank}, the other two operations do not change squared
ranks modulo $p$ either.

Consequently, after any sequence of allowed operations, the squared rank of
each basis vector is still congruent to $1$ modulo $p$.
\end{proof}

\section{Stronger numerical obstruction modulo $p^2$}\label{sec:mod-p2}

Besides the rank invariant, there is another obstruction preserved throughout
the standard orbit. Suppose $p=n+1$ is prime; then the entries of the standard
Gram matrix above the diagonal are not only divisible by $p$: their $p$-adic
order is exactly one. This property is preserved by mutations.

\begin{lemma}\label{lem:std-mod-p2}
Let $p$ be a prime and let $X=\bP^{p-1}$. In the standard basis
$e_i=[\cO(i)]$, $0\le i\le p-1$, we have
\[
\langle e_a,e_{a+d}\rangle
=
\binom{p-1+d}{p-1}
\equiv p d^{-1}\pmod {p^2}
\]
for $d=1,\dots,p-1$, where $d^{-1}$ denotes the inverse of $d$ modulo $p$. In
particular, all entries of the standard Gram matrix above the diagonal are
nonzero modulo $p^2$.
\end{lemma}

\begin{proof}
We have
\[
\binom{p-1+d}{p-1}
=\binom{p-1+d}{d}
=\frac{p}{d}\binom{p-1+d}{d-1}.
\]
Since
\[
\binom{p-1+d}{d-1}\equiv 1\pmod p,
\]
we obtain
\[
\frac{1}{p}\binom{p-1+d}{p-1}\equiv d^{-1}\pmod p.
\qedhere
\]
\end{proof}

\begin{proposition}\label{prop:no-zero-upper-standard-orbit}
Let $p$ be a prime and let $X=\bP^{p-1}$. Suppose a numerical exceptional basis
lies in the orbit of the standard basis under mutations, sign changes, and
isometries of the Euler form; then all entries of its Gram matrix above the
diagonal are nonzero modulo $p^2$. Consequently, any basis from this orbit does not contain two mutually orthogonal vectors.
\end{proposition}

\begin{proof}
It suffices to show that the following property is preserved throughout the
orbit:
\[
\langle e_i,e_j\rangle\in p\bZ\setminus p^2\bZ
\qquad \text{for all } i<j.
\]
The standard basis has this property by Lemma~\ref{lem:std-mod-p2}.

Consider the right mutation of a pair $(e_i,e_{i+1})$ and put
$m=\langle e_i,e_{i+1}\rangle$. By assumption, $m\in p\bZ\setminus p^2\bZ$.
The new pair has the form $(e_{i+1},f)$, where
\[
f=e_i-m e_{i+1}.
\]
The new entry above the diagonal inside this pair is
\[
\langle e_{i+1},f\rangle=-m;
\]
hence this entry is also nonzero modulo $p^2$. If $k<i$, then
\[
\langle e_k,f\rangle
=
\langle e_k,e_i\rangle-m\langle e_k,e_{i+1}\rangle
\equiv
\langle e_k,e_i\rangle
\pmod {p^2}
\]
because both factors of the second term are divisible by $p$. The right-hand
side is nonzero modulo $p^2$. Similarly, if $k>i+1$, then
\[
\langle f,e_k\rangle
=
\langle e_i,e_k\rangle-m\langle e_{i+1},e_k\rangle
\equiv
\langle e_i,e_k\rangle
\pmod {p^2},
\]
and the right-hand side is again nonzero modulo $p^2$. All other entries above
the diagonal do not change. The left (inverse) mutation is treated in the same
way.

A sign change multiplies the corresponding row and column by $-1$. An isometry
does not change the Gram matrix at all. Therefore the property holds
throughout the standard orbit.

In a numerical exceptional basis, all entries below the diagonal are zero.
Hence two basis vectors are mutually orthogonal if and only if the
corresponding entry above the diagonal is zero. In the standard orbit, this is
impossible.
\end{proof}

\section{Counterexample for $\bP^4$}\label{sec:p4}

Let $X=\bP^4$. In this case $p=5$ and the Gram matrix of the standard basis
\[
([\cO],[\cO(1)],[\cO(2)],[\cO(3)],[\cO(4)])
\]
is
\[
B=
\begin{pmatrix}
1&5&15&35&70\\
0&1&5&15&35\\
0&0&1&5&15\\
0&0&0&1&5\\
0&0&0&0&1
\end{pmatrix}.
\]
Consider the matrix
\[
P=
\begin{pmatrix}
16&-283&-239&-208&-59\\
-44&888&739&656&192\\
33&-951&-766&-708&-219\\
4&344&248&261&93\\
-8&-9&8&-9&-9
\end{pmatrix}.
\]
Denote by $f_0,\dots,f_4$ the vectors of $K_0(X)$ given by the columns of $P$
in the standard basis. A direct computation gives
\[
{}^tPBP=
\begin{pmatrix}
1&-10&-10&-10&-5\\
0&1&0&0&5\\
0&0&1&5&10\\
0&0&0&1&5\\
0&0&0&0&1
\end{pmatrix},
\qquad
\det P=-1.
\]
Therefore $(f_0,\dots,f_4)$ is a numerical exceptional basis.

Since the rank of a vector is the sum of its coordinates in the standard basis,
we have
\[
(\rk f_0,\rk f_1,\rk f_2,\rk f_3,\rk f_4)
=
(1,-11,-10,-8,-2).
\]
Modulo $5$, these ranks are congruent to
\[
(1,4,0,2,3)
\]
and their squares are congruent to
\[
(1,1,0,4,4).
\]
In particular,
\[
\rk(f_2)^2\equiv0\not\equiv1\pmod5.
\]

\begin{theorem}\label{thm:p4-nontransitive}
For $\bP^4$, the action of mutations, sign changes, and isometries of the Euler
form on numerical exceptional bases is not transitive.
\end{theorem}

\begin{proof}
By Proposition~\ref{prop:rank-invariant}, every basis $(e_0,\dots,e_4)$ in the
standard orbit satisfies $\rk(e_i)^2\equiv1\pmod5$ for all $i$. The basis
$(f_0,\dots,f_4)$ constructed above violates this congruence for $i=2$.
Therefore this basis does not lie in the standard orbit; hence the action is
not transitive.
\end{proof}

The same matrix gives a second proof, independent of the rank invariant.
Namely, the Gram matrix of $(f_0,\dots,f_4)$ has zero entries above the
diagonal:
\[
\langle f_1,f_2\rangle=\langle f_2,f_1\rangle=0,
\qquad
\langle f_1,f_3\rangle=\langle f_3,f_1\rangle=0.
\]
Thus this numerical exceptional basis contains mutually orthogonal pairs of
vectors in $K_0(\bP^4)$. By Proposition~\ref{prop:no-zero-upper-standard-orbit},
any basis from the standard orbit does not contain such pairs. Hence the basis
$(f_0,\dots,f_4)$ does not lie in the standard orbit for this reason as well.

\section{Counterexample for $\bP^6$}\label{sec:p6}

A similar computation gives a counterexample for $\bP^6$. In this case $p=7$.
Let $B_6$ be the Gram matrix of the standard basis
\[
([\cO],[\cO(1)],\dots,[\cO(6)]),
\]
i.e.,
\[
(B_6)_{ij}=
\begin{cases}
\binom{6+j-i}{6}, & i\le j,\\
0, & i>j,
\end{cases}
\qquad 0\le i,j\le 6.
\]
Consider the matrix
\[
P_6=
\begin{pmatrix}
-11&-6&-14&2&0&26&-2\\
60&37&95&-19&-14&-120&10\\
-134&-65&-185&58&88&168&-18\\
146&-7&26&-76&-216&12&12\\
-82&94&196&50&243&-194&0\\
22&-74&-168&-16&-128&144&-3\\
-2&18&42&2&26&-35&1
\end{pmatrix}.
\]
Then
\[
{}^tP_6B_6P_6=
\begin{pmatrix}
1&7&49&-35&-56&686&-42\\
0&1&7&-7&-7&294&0\\
0&0&1&-7&0&574&14\\
0&0&0&1&0&-84&0\\
0&0&0&0&1&28&7\\
0&0&0&0&0&1&-21\\
0&0&0&0&0&0&1
\end{pmatrix},
\qquad
\det P_6=1.
\]
Therefore the vectors $f_0,\dots,f_6$ given by the columns of $P_6$ in the
standard basis form a numerical exceptional basis. The ranks of these vectors
are the column sums of $P_6$:
\[
(\rk f_0,\dots,\rk f_6)=(-1,-3,-8,1,-1,1,0).
\]
Modulo $7$, the squares of these ranks are congruent to
\[
(1,2,1,1,1,1,0).
\]
In particular,
\[
\rk(f_6)^2\equiv0\not\equiv1\pmod7,
\qquad
\rk(f_1)^2\equiv2\not\equiv1\pmod7.
\]
By Proposition~\ref{prop:rank-invariant}, in every basis of the standard orbit
all basis vectors have squared ranks congruent to $1$ modulo $7$.
Therefore the numerical exceptional basis $(f_0,\dots,f_6)$ does not lie in the
standard orbit either.

\section{Polishchuk's obstruction and unliftable classes}\label{sec:polishchuk}

Let us explain why the vectors constructed above should be regarded as purely
numerical classes. We use the following consequence of Polishchuk's
theorem~\cite[Theorem~1.3(iii)]{Polishchuk}.

\begin{theorem}[Polishchuk]\label{thm:polishchuk-special}
Let $p$ be a prime and let $X=\bP^{p-1}_{\mathbb C}$. If
$E\in D^b(\operatorname{Coh}X)$ is an exceptional object, then
\[
[E]\equiv \pm[\cO(a)]
\quad \text{in } K_0(X)\otimes\bF_p
\]
for some $a\in\bZ/p\bZ$.
\end{theorem}

Consider the standard basis $e_i=[\cO(i)]$, $0\le i\le p-1$. The classes of
line bundles in $K_0(\bP^{p-1})\otimes\bF_p$ are exactly the vectors
$e_0,\dots,e_{p-1}$. Indeed, modulo $p$ we have $[\cO(a+p)]=[\cO(a)]$.
Therefore, by Theorem~\ref{thm:polishchuk-special}, the class of an exceptional
object is congruent modulo $p$ to one of the vectors $\pm e_i$.

Put
\[
\mathcal L_p=\{\pm e_0,\dots,\pm e_{p-1}\}
\subset K_0(\bP^{p-1})\otimes\bF_p.
\]
We say that a vector $x\in K_0(\bP^{p-1})$ is \emph{Polishchuk-obstructed} if
the reduction of $x$ modulo $p$ does not belong to $\mathcal L_p$. By
Theorem~\ref{thm:polishchuk-special}, a Polishchuk-obstructed vector cannot be
the class of an exceptional object; we say that such a vector is
\emph{unliftable}.

\begin{proposition}[preservation of unliftability]
\label{prop:nonliftability-stable}
Let $X=\bP^{p-1}_{\mathbb C}$, where $p$ is a prime, and let
$\mathcal E=(e_0,\dots,e_{p-1})$ be a numerical exceptional basis such that
\[
\langle e_i,e_j\rangle\equiv\delta_{ij}\pmod p
\]
and each $e_i$ is Polishchuk-obstructed. Then every basis obtained from
$\mathcal E$ by numerical mutations and sign changes has the same two
properties. Every isometry $g$ of the Euler form also preserves these two
properties provided that the reduction of $g$ modulo $p$ preserves the set
$\mathcal L_p$.
Consequently, no vector of any basis in this orbit is the class of an
exceptional object of $D^b(\operatorname{Coh}X)$.
\end{proposition}

\begin{proof}
Consider a pair $(e_i,e_{i+1})$ and put
$m=\langle e_i,e_{i+1}\rangle$. By assumption, $m\equiv0\pmod p$. The right
numerical mutation replaces this pair by
\[
(e_{i+1},\, e_i-m e_{i+1});
\]
the left mutation is similar. Since $m$ is divisible by $p$, the new pair
coincides with $(e_{i+1},e_i)$ modulo $p$ up to the sign convention. Therefore
a mutation only permutes the reductions of the basis vectors modulo $p$. In
particular, if the reductions of all basis vectors lie outside $\mathcal L_p$
before the mutation, then the same is true after the mutation. The proof of
Lemma~\ref{lem:mutation-rank} shows that the congruence
$\langle e_i,e_j\rangle\equiv\delta_{ij}\pmod p$ is also preserved; therefore
the argument can be repeated after each subsequent mutation.

A sign change replaces a reduction by its negative and preserves both
$\mathcal L_p$ and its complement. An isometry preserving $\mathcal L_p$
modulo $p$ also preserves the complement of $\mathcal L_p$. Induction on the
length of the sequence of operations proves the first two assertions; the last
assertion then follows from Theorem~\ref{thm:polishchuk-special}.
\end{proof}

\begin{remark}
The assumption on isometries is essential. Polishchuk's theorem singles out
the set of classes of line bundles modulo $p$. Mutations and sign changes
automatically preserve the corresponding obstruction provided the Gram matrix
$B$ satisfies $B\equiv I\pmod p$. However, an arbitrary integral isometry of
the Euler lattice need not preserve $\mathcal L_p$. For this reason,
Proposition~\ref{prop:nonliftability-stable} is stated for the orbit under
mutations and sign changes or for its extension by isometries preserving
$\mathcal L_p$.
\end{remark}

For the matrix $P$ of Section~\ref{sec:p4}, reduction modulo $5$ gives
\[
P\equiv
\begin{pmatrix}
1&2&1&2&1\\
1&3&4&1&2\\
3&4&4&2&1\\
4&4&3&1&3\\
2&1&3&1&1
\end{pmatrix}
\pmod5.
\]
No column of this matrix has the form $\pm e_i$. Therefore none of the vectors
$f_0,\dots,f_4$ is the class of an exceptional object in
$D^b(\operatorname{Coh}\bP^4)$. This is stronger than the rank obstruction
alone. For example, the ranks of $f_0$ and $f_1$ are congruent to $\pm1$
modulo $5$; however, the reductions of $f_0$ and $f_1$ modulo $5$ do not belong
to $\mathcal L_5$, i.e., these reductions are not classes of line bundles, even
up to sign.

For the matrix $P_6$ of Section~\ref{sec:p6}, reduction modulo $7$ gives
\[
P_6\equiv
\begin{pmatrix}
3&1&0&2&0&5&5\\
4&2&4&2&0&6&3\\
6&5&4&2&4&0&3\\
6&0&5&1&1&5&5\\
2&3&0&1&5&2&0\\
1&3&0&5&5&4&4\\
5&4&0&2&5&0&1
\end{pmatrix}
\pmod7.
\]
Again, no column has the form $\pm e_i$. Hence none of the vectors
$f_0,\dots,f_6$ is the class of an exceptional object in
$D^b(\operatorname{Coh}\bP^6)$; this includes the vectors with ranks congruent
to $\pm1$ modulo $7$.

Thus the mutually orthogonal pairs in the numerical basis for $\bP^4$ are
``ghost'' pairs: these pairs exist in the Euler lattice $K_0$ but do not lift
to pairs of exceptional objects in the derived category. Moreover, by
Proposition~\ref{prop:nonliftability-stable}, every vector of every basis
obtained from these bases by mutations and sign changes is also unliftable; the
same is true after applying isometries of the Euler form preserving
$\mathcal L_p$.

\section{Absence of mutually orthogonal exceptional objects on
$\bP^{p-1}$}\label{sec:orthogonal}

We now prove the categorical counterpart of the preceding numerical
observation.

\begin{theorem}\label{thm:no-orthogonal-exceptional}
Let $p$ be a prime and let $X=\bP^{p-1}_{\mathbb C}$. Then
$D^b(\operatorname{Coh}X)$ does not contain a pair of mutually orthogonal
exceptional objects, i.e., there are no exceptional objects $E,F$ such that
\[
\Ext^\bullet(E,F)=0,
\qquad
\Ext^\bullet(F,E)=0.
\]
\end{theorem}

\begin{proof}
Put
\[
e_i=[\cO(i)],
\qquad i=0,\dots,p-1.
\]
These vectors form a basis of $K_0(\bP^{p-1})$. The matrix of the Euler form in
this basis is
\[
B_{ij}=\chi(e_i,e_j)=
\begin{cases}
\binom{p-1+j-i}{p-1}, & i\le j,\\
0, & i>j.
\end{cases}
\]
By Lemma~\ref{lem:std-mod-p2}, we have
\[
B=I+pN_1,
\]
where all entries of $N_1$ below the diagonal are zero and
\[
(N_1)_{a,a+d}\equiv d^{-1}\pmod p,
\qquad d=1,\dots,p-1.
\]
Therefore,
\[
(N_1)_{ab}-(N_1)_{ba}\not\equiv0\pmod p
\qquad\text{whenever } a\ne b.
\]

Assume the converse: let $E$ and $F$ be mutually orthogonal exceptional
objects. Put $x=[E]$ and $y=[F]$; then
\[
\chi(x,y)=0,
\qquad
\chi(y,x)=0.
\]
By Theorem~\ref{thm:polishchuk-special}, there exist
$a,b\in\{0,\dots,p-1\}$, signs $\varepsilon,\delta\in\{\pm1\}$, and elements
$u,v\in K_0(\bP^{p-1})$ such that
\[
x=\varepsilon e_a+p u,
\qquad
y=\delta e_b+p v.
\]
Since $B\equiv I\pmod p$, we have
\[
\chi(x,y)\equiv \varepsilon\delta\,\delta_{ab}\pmod p.
\]
If $a=b$, then the right-hand side is nonzero modulo $p$; this contradicts the
equality $\chi(x,y)=0$. Therefore $a\ne b$.

Now consider the skew-symmetric part of the Euler form. We have
\[
B-B^t=p(N_1-N_1^t);
\]
hence,
\[
\chi(x,y)-\chi(y,x)
=x^t(B-B^t)y
\equiv
p\varepsilon\delta\, e_a^t(N_1-N_1^t)e_b
\pmod {p^2},
\]
i.e.,
\[
\chi(x,y)-\chi(y,x)
\equiv
p\varepsilon\delta\bigl((N_1)_{ab}-(N_1)_{ba}\bigr)
\pmod {p^2}.
\]
Since $a\ne b$, the factor $(N_1)_{ab}-(N_1)_{ba}$ is nonzero modulo $p$. Thus
the right-hand side is nonzero modulo $p^2$. On the other hand,
$\chi(x,y)-\chi(y,x)=0$. This contradiction proves the theorem.
\end{proof}

\section{Limits of applicability of the argument modulo $p^2$}\label{sec:boundary}

Polishchuk's theorem also applies to
\[
X=\bP^{p^r-1}_{\mathbb C},
\qquad r>1:
\]
modulo $p$, the class of an exceptional object still coincides with the class
of a line bundle up to sign. However, the proof of
Theorem~\ref{thm:no-orthogonal-exceptional} uses a stronger property specific
to the case $r=1$.

Let $n=p^r$ and let $X=\bP^{n-1}_{\mathbb C}$. Then
\[
\chi(\cO(a),\cO(a+1))
=
\binom{n}{n-1}
=n
=p^r.
\]
If $r=1$, then $p^r/p=1$ is nonzero modulo $p$. If $r>1$, then
\[
\frac{p^r}{p}=p^{r-1}\equiv0\pmod p.
\]
Hence the form
\[
\frac{\chi(x,y)-\chi(y,x)}{p}\pmod p
\]
no longer distinguishes the line bundles $\cO(a)$ and $\cO(a+1)$.
More generally, for $1\le d\le p^r-1$ we have
\[
v_p\!\left(\binom{p^r-1+d}{p^r-1}\right)=r-v_p(d).
\]
Indeed,
\[
\binom{p^r-1+d}{p^r-1}
=
\frac{p^r}{d}\prod_{j=1}^{d-1}\frac{p^r+j}{j}.
\]
Each factor of this product is a $p$-adic unit because $v_p(p^r+j)=v_p(j)$ for
$1\le j<p^r$. Therefore the skew-symmetric part may first appear at a higher
$p$-adic level. Consequently, the congruence
\[
[E]\equiv\pm[\cO(m)]\pmod p
\]
alone does not suffice to extend the above proof to all
$\bP^{p^r-1}$ with $r>1$; such an extension requires additional $p$-adic
information about the class $[E]$.

\end{document}